\documentclass{article}
\usepackage[utf8]{inputenc}
\usepackage{amsmath, amsthm, amssymb, amsfonts}
\usepackage{bm}
\usepackage{dsfont}
\usepackage[utf8]{inputenc}
\usepackage{rotate}
\usepackage{tikz}
\usepackage{tikz-cd}
\usepackage[arrow,matrix,curve]{xy} 	
\usepackage{xcolor}
\usepackage{todonotes}
\usepackage{alltt} % for coding

\theoremstyle{plain}
\newtheorem{theorem}{Theorem}[section]
\newtheorem{proposition}[theorem]{Proposition}
\newtheorem{lemma}[theorem]{Lemma}

\theoremstyle{definition}

\newtheorem{remark}[theorem]{Remark}

\newcommand{\norm}[1]{\left\lVert#1\right\rVert}

\title{The Pair Correlation Function of Multi-Dimensional Low-Discrepancy Sequences with Small Stochastic Error Terms}
\author{Anja Schmiedt, Christian Wei\ss{}}
\date{\today}

\begin{document}

\maketitle

\begin{abstract}  In any dimension $d \geq 2$, there is no known example of a low-discrepancy sequence which possess Poisssonian pair correlations. This is in some sense rather surprising, because low-discrepancy sequences always have $\beta$-Poissonian pair correlations for all $0 < \beta < \tfrac{1}{d}$ and are therefore arbitrarily close to having Poissonian pair correlations (which corresponds to the case $\beta = \tfrac{1}{d}$). In this paper, we further elaborate on the closeness of the two notions. We show that $d$-dimensional Kronecker sequences for badly approximable vectors $\vec{\alpha}$ with an arbitrary small uniformly distributed stochastic error term generically have $\beta = \tfrac{1}{d}$-Poissonian pair correlations.
\end{abstract}

\section{Introduction}

According to a famous theorem of Weyl, \cite{Wey16}, a sequence $(\vec{x}_n)_{n \in \mathbb{N}}$ in $[0,1]^{d}$ is uniformly distributed if and only if for all vectors $\vec{r} \in \mathbb{Z}^d \setminus \{ \vec{0} \}$ it holds that
$$\lim_{N \to \infty} \frac{1}{N} \sum_{j=1}^N e(\langle \vec{r}, \vec{x}_{j}\rangle) = 0,$$
where $e(\cdot) := \exp(2 \pi i \cdot)$ and $\langle \cdot,\cdot\rangle$ denotes the $\ell^2$ scalar product on $\mathbb{R}^d$. The result established an important link between uniform distribution theory and exponential sums, which are a central tool in analytic number theory. A classical way to quantify the degree of uniformity of $(x_n)_{n \in \mathbb{N}}$ is the discrepancy, which is defined by
$$D_N(x_n) := \sup_{B \subset [0,1)^{d}} \left| \frac{1}{N}\# \left( \left\{x_i | 1 \leq i \leq N \right\} \cap B \right)
%\frac{\sum_{i=1}^N \mathds{1}_B(p_n) }{N} 
- \lambda_d(B) \right|,$$
where the supremum is taken over all boxes $B = [a,b) = [a_1,b_1) \times \ldots \times [a_d,b_d) \subset [0,1)^{d}$ and $\lambda_d(\cdot)$ denotes the $d$-dimensional Lebesgue measure. It is well-known that a sequence $(x_n)_{n \in \mathbb{N}}$ is uniformly distributed if and only if $D_N(x_n) \to 0$ for $N \to \infty$, compare \cite{Nie92}. The so-called Koksma-Erd\"os-Turan  inequality may be regarded as a quantitative version of Weyl's theorem and states that for any positive $m \in \mathbb{N}$ we have
$$D_N(x_n) \leq C_d \left(\frac{1}{m} + \sum_{0 < \norm{h}^{\infty} \leq m} \frac{1}{r(h)} \left| \frac{1}{N} \sum_{n=1}^N e(\langle h, x_n \rangle) \right|\right),$$
where $C_d$ is a constant that only depends on the dimension, $\norm{h}^\infty=\max_{1 \leq j \leq d} |h_j|$ and $r(h)= \prod_{j=1}^d \max(|h_j|,1)$, see \cite{KN74}, Chapter 2.\\[12pt]  
If a sequence $(x_n)_{n \in \mathbb{N}} \in [0,1)^{d}$ exhibits a discrepancy of order
\begin{equation*}
D_N(x_n) = \mathcal{O}(N^{-1}(\log N)^{d}),
\end{equation*}
where $\mathcal{O}(\cdot)$ denotes the Landau symbol, then it is called a low-discrepancy sequence. In fact, this is the best possible rate of convergence in dimension $d=1$ by the work of Schmidt, \cite{Sch72}, and it is widely conjectured that this is the optimal bound in arbitrary dimension. A wide variety of low-discrepancy sequences is known, ranging in dimension $d=1$ from more classical ones, like van der Corput sequences and Kronecker sequences, see e.g. \cite{KN74}, to more recent ones as the classes of examples in \cite{Car12} or \cite{Wei19}. In higher dimensions, there often exist multi-dimensional versions of the mentioned examples, see e.g. \cite{Nie92}. \\[12pt] 
In this paper, we are mainly interested in Kronecker sequences which are for $\alpha \in \mathbb{R} \setminus \mathbb{Q}$ defined by $x_n := \{ n \alpha\}$, where $\{ \cdot \}$ denotes the fractional part of a number. If $\vec{\alpha} \in \mathbb{R}^d$, then $\{ n \vec{\alpha} \}$ is defined component-wise.\\[12pt]
Another concept to quantify the degree of uniformity of a sequence $(x_n)_{n \in \mathbb{N}}$ was introduced by Rudnik and Sarnak for dimension $d=1$ in \cite{RS98} and later on generalized to the $d$-dimensional setting in \cite{HKL19}. It is based on the pair correlation function defined by
$$F_{N,d}(s) := \frac{1}{N} \# \left\{ 1 \leq k \neq l \leq N \ : \ \left\| x_k - x_l\right\|_{\infty} \leq \frac{s}{N^{1/d}} \right\},$$
where $\left\| \cdot \right\|$ is the distance of a number from its nearest integer in the $1$-dimensional setting and
$$\norm{x}_{\infty} := \max(\norm{x_1},\norm{x_2},\ldots,\norm{x_d}).$$ 
The sequence $(x_n)_{n \in \mathbb{N}}$ is said to have Poissionian pair correlations if
$$ \lim_{N \to \infty} F_{N,d}(s) = (2s)^d$$
for all $s> 0$. This definition was generalized by Nair and Policott in \cite{NP07} to $\beta$-Poissonian pair correlations in the one-dimensional setting and later on transferred to higher dimensions as well, see e.g. \cite{Wei22a} (where the case $\beta = \tfrac{1}{d}$ corresponds to the original Poissonian pair correlation property). A sequence $(x_n)_{n \in \mathbb{N}}$ in $[0,1)^{d}$ has $\beta$-PPC for $0 \leq \beta \leq \tfrac{1}{d}$ if
$$\lim_{N \to \infty} F_{N,d} ^\beta(s) := \lim_{N \to \infty} \frac{1}{N^{2-\beta}} \# \left\{ 1 \leq k \neq l \leq N \ : \ \norm{x_k - x_l}_{\infty} \leq \frac{s}{N^\beta} \right\} = (2s)^d$$
for all $s > 0$. In dimension $d=1$, a sequence with $\beta_1$-PPC also has $\beta_2$-PPC for all $0 \leq \beta_2 < \beta_1 < 1$ according to \cite{HZ21}, Theorem 4. Generalizing results from \cite{GL17} and \cite{ALP18}, it was shown in \cite{Ste20} that $\beta$-PPC imply uniform distribution for any $\beta \geq 0$. Vice versa, a sequence of independent, uniformly distributed random vectors $(\vec{X}_n)_{n \in \mathbb{N}}$ generically has $\beta$-PPC for all $0 \leq \beta \leq \tfrac{1}{d}$. In fact, an even stronger statement is known for low-discrepancy sequences.
\begin{theorem}[Theorem 1.1, \cite{Wei22a}] \label{thm:ldppc} Every low-discrepancy sequence $(x_n)_{n \in \mathbb{N}}$ in $[0,1]^d$ has $\beta$-PPC for all $0 \leq \beta < \tfrac{1}{d}$.
\end{theorem}
However, the above mentioned examples of low-discrepancy in dimension one all do not have Poissonian pair correlations, see \cite{LS20} and \cite{WS22}. More generally, it seems to be challenging to find explicit examples of sequences with Poissonian pair correlations, and only few ones are known by now, see e.g. \cite{BMV15}, \cite{LST21}.\\[12pt] 
Theorem~\ref{thm:ldppc} may however be interpreted that low-discrepancy sequences fail to have $\tfrac{1}{d}$-PPC as closely as possible. For a class of Kronecker sequences, we further elaborate on this interpretation by showing that an arbitrarily small stochastic distortion of these sequences generically implies $\tfrac{1}{d}$-PPC, see Theorem~\ref{thm:main} below. Recall that a number $\alpha \in \mathbb{R}$ is called badly approximable if there exists a $c = c_{\alpha} > 0$ such that for all $p \in \mathbb{Z}, q \in \mathbb{Z} \setminus \{0\}$ it holds that
$$\left| \alpha - \frac{p}{q} \right| > \frac{c}{q^2}$$
or equivalently
$$\liminf_{q \to \infty} \left\| q\alpha  \right\| = c > 0.$$
It is a standard result from analytic number theory that $\alpha$ is badly approximable if and only if all partial quotients in the continued fraction expansion of $\alpha$ are bounded.
\begin{theorem} \label{thm:main} Let $\vec{\alpha} \in \mathbb{R}^d$ be a vector consisting of (over $\mathbb{Q}$) linearly independent, badly approximable numbers and let $(\vec{X}_i)_{i \in \mathbb{N}}$ be a sequence of independent, identically distributed random vectors which are uniformly distributed on $[0,1]^d$. Furthermore let $\varepsilon > 0$ be arbitrary. Then the sequence of random vectors $(\vec{Y}_n)_{n \in \mathbb{N}} := \left(\left\{n \vec{\alpha} + \varepsilon \vec{X}_n\right\}\right)_{n \in \mathbb{N}}$ generically  has Poissonian pair correlations.
\end{theorem}
In our view, Theorem~\ref{thm:main} yields new insight in three different ways: First, it adds a new aspect to the interpretation that certain Kronecker sequences (as prominent examples of low-discrepancy sequences) are as close as possible to having Poissonian pair correlations. Second, it constitutes the first non-trivial examples of a pair of two sequences in arbitrary dimension which both do not posses $\tfrac{1}{d}$-PPC but their sum does.\footnote{A trivial example would be the following: consider a sequence $(x_n)_{n \in \mathbb{N}}$ with $\tfrac{1}{d}$-PPC and split it up into two sequences $(y_n)_{n \in \mathbb{N}}, (z_n)_{n \in \mathbb{N}}$, where $y_n$ is equal to $x_n$ for even indices and $z_n$ is equal to $x_n$ for odd indices while all other elements of $y_n$ and $z_n$ are $\vec{0}.$} Third, our proof of Theorem~\ref{thm:main} in Section~\ref{sec:sums}, whose structure is inspired by the ones of \cite{Mar07}, Theorem 2.3, and \cite{HKL19}, Theorem 1, strongly refines the technique used in the mentioned references. Indeed, we even conjecture that Theorem~\ref{thm:main} might hold for any low-discrepancy sequence in arbitrary dimension for the following two reasons: in dimension $d=1$, Theorem~\ref{thm:main} is a special instance of Theorem~1.2 in \cite{LR21} which covers a wider class of examples in the one-dimensional setting. Second, despite that the proof of Theorem~\ref{thm:main} relies on the properties of Kronecker sequences and badly approximable numbers, the conjecture seems to be reasonable, because we can show a relevant intermediate step of the proof for arbitrary low-discrepancy sequences. As this result is of independent interest, we state it here separately.
\begin{proposition} \label{prop:expected} Let $(\vec{z}_n)_{n \in \mathbb{N}}$ be a low-discrepancy sequence in $[0,1]^d$, $\varepsilon > 0$ and $(\vec{X}_n)_{n \in \mathbb{N}}$ be a sequence of independent random vectors which are uniformly distributed on $[0,1]^{d}$. For arbitrary $s>0$, let $F_{N,d}(s)$ be the pair correlation function of $(\vec{Y}_n)_{n \in \mathbb{N}} = (\{\vec{z}_n + \varepsilon \vec{X}_n\})_{n \in \mathbb{N}}$. Then the expected value $\mathbb{E}[F_{N,d}(s)]$ converges to $(2s)^d$ for $N \to \infty$.
\end{proposition}
The subsequent sections are organized as follows: in Section~\ref{sec:preparatory_results}, we collect some rather general results in four lemmas. They are all important for proving Theorem~\ref{thm:main} but can be formulated separately and are used several times in the remainder of this paper. Afterwards, we give the rather lengthy proof of Theorem~\ref{thm:main} (and of Proposition~\ref{prop:expected}) in Section~\ref{sec:sums}.
\paragraph{Acknowledgment.} The authors would like to thank Stefan Steinerberger for fruitful discussions on important aspects of this paper.

\section{Preparatory Results} \label{sec:preparatory_results}

In this section, we collect the mainly technical results which we need to prove our main theorem. Since we are in our context not interested in numerically optimal bounds, we do not try to optimize the constants here. This might be up for future research. The first lemma calculates an expected value which will play a central role in the remainder of this paper and establishes the main difference and challenge in comparison to \cite{Mar07} and \cite{HKL19} due to the involvement of $\sin$-terms. It suffices to only consider the one-dimensional case here.

\begin{lemma} \label{lem:integral} Let $r,r' \in \mathbb{Z} \setminus \{ 0 \}$. Furthermore let $(X_n)_{n \in \mathbb{N}}$ be a sequence of independent, identically distributed random variables, which are uniformly distributed on $[0,1]$. In the case $r \neq \pm r'$ it follows for $k,l,m,n \in \mathbb{N}$ with $k \neq l$ and $m \neq n$ and any $\varepsilon > 0$ that
\begin{align*}\mathbb{E}[e&(r\varepsilon(X_k-X_l) +  r' \varepsilon(X_m-X_n))]
\\ & = \begin{cases} \frac{1}{(r+r')\pi \varepsilon} \sin((r+r')\pi \varepsilon) \frac{1}{r \varepsilon} \sin(r \pi \varepsilon) \frac{1}{r' \varepsilon} \sin(r' \pi \varepsilon) & \text{if} \ k = m,  l \neq n\\
\frac{1}{(r+r')\pi \varepsilon} \sin((r+r')\pi \varepsilon) \frac{1}{r \varepsilon} \sin(r \pi \varepsilon) \frac{1}{r' \varepsilon} \sin(r' \pi \varepsilon)  & \text{if} \ k  \neq m,  l = n\\
\frac{1}{(r-r')\pi \varepsilon} \sin((r-r')\pi \varepsilon) \frac{1}{r \varepsilon} \sin(r \pi \varepsilon) \frac{1}{r' \varepsilon} \sin(r' \pi \varepsilon)& \text{if} \ k = n, l \neq m\\
\frac{1}{(r-r')\pi \varepsilon} \sin((r-r')\pi \varepsilon) \frac{1}{r \varepsilon} \sin(r \pi \varepsilon) \frac{1}{r' \varepsilon} \sin(r' \pi \varepsilon) & \text{if} \ k \neq n, l = m\\
\frac{1}{((r+r`)\pi\varepsilon)^2} \sin((r+r')\pi\varepsilon)^2 & \text{if} \ k  = m,  l = n\\
\frac{1}{((r-r`)\pi\varepsilon)^2} \sin((r-r')\pi\varepsilon)^2 & \text{if} \ k=n, l = m\\
\frac{1}{(r \pi \varepsilon)^2} \sin(r\pi \varepsilon)^2 \frac{1}{(r' \pi \varepsilon)^2} \sin(r'\pi \varepsilon)^2 & \text{else}.
\end{cases}
\end{align*}
If $r=r'$, then all expressions of the form $\frac{1}{ (r-r') \pi \varepsilon} \sin((r-r')\pi \varepsilon)$ need to be replaced by $1$ in the formulae above. Similarly, if $r=-r'$, then all expressions of the form $\frac{1}{ (r+r') \pi \varepsilon} \sin((r+r')\pi \varepsilon)$ need to be replaced by $1$.
\end{lemma}
\begin{proof} In the cases $k \neq m, l \neq n$ and $k=m, l=n$, the result follows by simple integration and trigonometrical arguments. If $k=m, l \neq n$, integration leads to
\begin{align*}
    \mathbb{E}[e&(r \varepsilon(X_k-X_l) +  r' \varepsilon(X_m-X_n))] = \\
    & \frac{\sin((r+r') \pi \varepsilon}{(r+r') \pi \varepsilon} \left( \cos((r+r') \pi \varepsilon) + i \sin((r+r') \pi \varepsilon) \right)\\
    & \times \frac{\sin(r \pi \varepsilon)}{r \pi \varepsilon}\left( \cos(r \pi \varepsilon) + i \sin(r \pi \varepsilon) \right)\\
    & \times \frac{\sin(r' \pi \varepsilon)}{r' \pi \varepsilon}\left( \cos(r' \pi \varepsilon) + i \sin(r' \pi \varepsilon) \right)
\end{align*}
Using the standard angel sum equations $\sin(x+y) = \sin(x) \cos(y) + \sin(y) \cos(x)$ and $\cos(x+y) = \cos(x)\cos(y) - \sin(x)\sin(y)$, we arrive the assertion follows. The remaining cases follow similarly.
\end{proof}
\begin{remark} \label{rem:doublesequence} If $r'=0$ and $r \neq 0$, then the expected value is equal to $\frac{\sin(r\pi\varepsilon)^2}{(r\pi\varepsilon)^2}$.
\end{remark}
There will also appear exponential sums of the Kronecker part of the sequence. This is also the point where we make use of the fact that $\vec{\alpha}$ consists of badly approximable components.
\begin{lemma} \label{lem:upperbound} For all vectors $\vec{\alpha} \in \mathbb{R}^{d}$ consisting of linearly independent, badly approximable numbers and $0 \leq \delta \leq \tfrac{1}{2}$, there exists $C_{\vec{\alpha},\delta} > 0$ such that for all $N \in \mathbb{N}$ and $\vec{r} \in \mathbb{N}^d$ we have
$$\left| \sum_{j=1}^N e(j\langle \vec{r},\vec{\alpha}\rangle) \right| \leq C_{\vec{\alpha},\delta} \min_{1 \leq i \leq d} r_i^{1/2-\delta}N^{1/2+\delta}$$
\end{lemma}
\begin{proof} Since the left hand side is always $\leq N$, the inequality holds automatically for $\min_{1 \leq i \leq d} r_i \geq N$. In order to prove the assertion for $\min_{1 \leq i \leq d} r_i < N$, we show the stronger inequality
$$\left| \sum_{j=1}^N e(j\langle\vec{r},\vec{\alpha}\rangle \right| \leq C_{\vec{\alpha}} \min_{1 \leq i \leq d} r_i.$$
We have
$$\left| \sum_{j=1}^N e(j\langle\vec{r},\vec{\alpha}\rangle) \right|  \leq \frac{2}{\left| 1 - \exp(2\pi i\langle\vec{r},\vec{\alpha}\rangle) \right|}.$$
As $\vec{\alpha}$ consists of badly approximable numbers, there exists a $\frac{1}{\sqrt{5}} \geq \tilde{C}_{\vec{\alpha}} > 0$ with $$\left|r_i \alpha_i - p\right| \geq \tilde{C}_{\vec{\alpha}} \frac{1}{r_i}, \quad 1 \leq i \leq d,$$
for all $r_i \in \mathbb{N}$ and $p \in \mathbb{Z}$. We may without loss of generality assume that even $\frac{1}{4d} > \tilde{C}_{\vec{\alpha}}$. Therefore,
$$\left|(1-\exp(2\pi i\langle\vec{r},\vec{\alpha}\rangle) \right| \geq \sin(2 \pi \tilde{C}_{\vec{\alpha}}\sum_{i=1}^d 1 / r_i) > \tilde{C}_{\vec{\alpha}} \pi \sum_{i=1}^d \frac{1}{r_i}$$
which implies the claim with $C_{\vec{\alpha}} = \frac{1}{\pi \tilde{C}_{\vec{\alpha}}}$. 
\end{proof}
Combining these results, we get the following inequality in dimension $d=1$, which will be used in the proof of Theorem~\ref{thm:main}.
\begin{lemma} \label{lem:doublesum} 
For all $0 < \delta \leq \tfrac{1}{2}$, all $\varepsilon > 0$ and all badly approximable $\alpha \in \mathbb{R}$ we have
\begin{align*}
    &  \sum_{r \in \mathbb{Z} \setminus \{ 0 \}} \left| \sum_{k=1}^N \sum_{\substack{l=1,\\ l \neq k}}^N E[e(r(\varepsilon(X_k-X_l) + k\alpha - l \alpha))]\right|\\
    & \leq C_{\alpha,\delta,\varepsilon} N^{1+\delta} 
    \end{align*}
for some $C_{\alpha,\delta,\varepsilon} \in \mathbb{R}$ which only depends on $\alpha,\delta$ and $\varepsilon$ but not on $r$.
\end{lemma}
\begin{proof}
By Remark~\ref{rem:doublesequence}, Lemma~\ref{lem:upperbound} and the fact that $\sum_{k=1}^\infty \tfrac{1}{r^{1+\delta}}$ is a convergent series, we obtain 
\begin{align*}
    &  \sum_{r \in \mathbb{Z} \setminus \{ 0 \}} \left| \sum_{k=1}^N \sum_{\substack{l=1,\\ l \neq k}}^N E[e(r(\varepsilon(X_k-X_l) + k\alpha - l \alpha))]\right|\\
    & \leq \tilde{C}_{\alpha,\delta} \sum_{r=1}^\infty \left| \frac{2}{\pi^2 r^2 \varepsilon^2} \sin(\pi r \varepsilon)^2 N^{1+\delta} r^{1-\delta}\right|\\
    & \leq \tilde{C}_{\alpha,\delta,\varepsilon} N^{1+\delta} \sum_{r=1}^\infty \left| \frac{\sin(\pi r \varepsilon)}{r^{1+\delta}} \right|\\
    & \leq \tilde{C}_{\alpha,\delta,\varepsilon} \zeta(1+\delta) N^{1+\delta},
    \end{align*}
where $\zeta(\cdot)$ denotes the Riemann zeta function. The assertion follows with $C_{\alpha,\delta,\varepsilon} = \tilde{C}_{\alpha,\delta,\varepsilon} \zeta(1+\delta)$.
\end{proof}
Finally, for the cases in Lemma~\ref{lem:integral}, where two of the variables $X_k,X_l,X_m,X_n$ coincide, we need to use a another one-dimensional bound on a certain series which is in our view also of independent interest.
\begin{lemma} \label{lem:shortsums} For all $r' \in \mathbb{Z} \setminus \{0 \}$ and all $0 \leq \sigma < 1$ we have
$$\sum_{\substack{r\in \mathbb{N},\\ r \neq \pm r'}} \left| \frac{1}{r} \right|^\sigma \frac{1}{|r+r'|^2} \leq (2+3\zeta(2)) \frac{1}{|r'|^\sigma}.$$
\end{lemma}
\begin{proof} We need to show that 
$$\sum_{\substack{r\in \mathbb{N}\\ r \neq \pm r'}} \left| \frac{r'}{r} \right|^\sigma \frac{1}{|r+r'|^2}$$
is uniformly bounded. At first we split the sum into three parts
$$\sum_{\substack{r < |r`|/2,\\ r \neq r'}} \left| \frac{r'}{r} \right|^\sigma \frac{1}{|r + r'|^2} + \sum_{\substack{|r`|/2 \leq r < |r'|,\\ r \neq r'}} \left| \frac{r'}{r} \right|^\sigma \frac{1}{|r + r'|^2} + \sum_{\substack{r \geq |r`|,\\ r \neq r'}} \left| \frac{r'}{r} \right|^\sigma \frac{1}{|r + r'|^2} %+ \sum_{r \geq |r`|, r \neq r'} \frac{1}{\left|r^{1/2}|r'|^{1/2}|\pm r^{3/2}|r|^{-1/2}\right|^2},
$$
The first sum consists of less than $|r'|/2$ terms of size at most $|r'|^\sigma / |r'/2|^2$ and is thus strictly bounded by $2$. The first factor in the second sum is at most $2^\sigma$ and thus the sum can be bounded by $2\zeta(2)$. Finally, for the third sum, the factor $\left| \frac{r'}{r} \right|$ is bounded by $1$ and therefore this sum can be bounded by $\zeta(2)$.
\end{proof}

\section{Sums of non-Poissonian Sequences} \label{sec:sums}

In this section we turn to the proof of Theorem~\ref{thm:main} that uses the intermediate result from Proposition~\ref{prop:expected} that does hold for arbitrary multi-dimensional low-discrepancy sequences. The structure of the proof of our main result is essentially similar to the one of Theorem~2.3 in \cite{Mar07} and Theorem 1 in \cite{HKL19}. However, we face additional technical challenges because the sequence under consideration has a deterministic and a stochastic part with a comparably complicated expected value, see Lemma~\ref{lem:integral}.\\[12pt]
A main step within the proof of Theorem~\ref{thm:main} is to apply Chebyshev's inequality in order to obtain convergence in probability. Therefore, Proposition~\ref{prop:expected} gives an estimate on the expected value of the pair correlation function first. We now come to its proof.
\begin{proof}[Proof of Proposition~\ref{prop:expected}]
By definition we have
\begin{align*}
  & \mathbb{E}\left(F_{N,d}(s)\right) =\\
  & \mathbb{E}\left( \frac{1}{N} \# \left\{ 1 \leq l \neq m \leq N \ : \ \norm{\vec{z}_l - \vec{z}_m + \varepsilon \cdot (\vec{X}_l - \vec{X}_m)} _{\infty} \leq \frac{s}{N^{1/d}} \right\} \right).
\end{align*}
If we regard the expression $\varepsilon (\vec{X}_l-\vec{X}_m)$ as the difference of two independent, uniformly distributed random variables on $[0,\varepsilon]^{d}$. Recall that for arbitrary measurable sets $A,B$, the convolution of their characteristic functions is given by $\mathds{1}_A * \mathds{1}_B(x) = \lambda_d(B\cap(x-A))$. Thus, the density of  $\varepsilon (\vec{X}_l-\vec{X}_m)$ equals
$$f(\vec{x}) = \frac{1}{\varepsilon^{d}} \prod_{i=1}^d \left( 1 - \left| \frac{x_i}{\varepsilon} \right| \ \right) \mathds{1}_{[-\varepsilon,\varepsilon]^{d}}(\vec{x}).$$
For fixed $1 \leq l \leq N$ we consider the expected value of
$$F_{N,d}^l(s) := \frac{1}{N} \# \left\{ 1 \leq m \leq N \ : \ m \neq l, \norm{\vec{z}_l - \vec{z}_m + \varepsilon \cdot (\vec{X}_l - \vec{X}_m) }_{\infty} \leq \frac{s}{N^{1/d}} \right\}.$$
Denote the $i$-th component of $\vec{z}_m$ by $z_m^{(i)}$ and likewise for $\vec{z}_l$. Then by the definition of the norm $\norm{\cdot}_{\infty}$ we obtain,\footnote{Note that the corresponding proof in \cite{HKL19} uses the fact that $\varepsilon =1$ when calculating the expected value. Therefore, the integration bounds only need to be considered modulo 1. This explains the technical difference between their proof and our.}
\begin{align*}
& \mathbb{E}(F_{N,d}^l(s)) = \\ & \frac{1}{N} \sum_{\substack{m \neq l,\\ \norm{z_m - z_l}_{\infty} \leq \frac{s}{N^{1/d}} + \varepsilon}}  \int_{z_m^{(1)}-z_l^{(1)}-s/N^{1/d}}^{z_m^{(1)}-z_l^{(1)}+s/N^{1/d}} \ldots \int_{z_m^{(d)}-z_l^{(d)}-s/N^{{1/d}}}^{z_m^{(d)}-z_l^{(d)}+s/N^{1/d}} f(x) \mathrm{d} x_1 \ldots \mathrm{d} x_d 
\end{align*}
The $d$-fold integral is equal to
\begin{align*}
& \left(\frac{-1}{2}\right)^d  \prod_{i=1}^d \left[ \left( 1 - \frac{z_m^{(i)}-z_l^{(i)}+s/N^{1/d}}{\varepsilon} \right)^2 - \left( 1 - \frac{z_m^{(i)}-z_l^{(i)}-s/N^{1/d}}{\varepsilon} \right)^2\right]\\
& = \frac{(2s)^d}{\varepsilon^d} \cdot \frac{1}{N} \prod_{i=1}^d \left(1 - \frac{z_m^{(i)}-z_l^{(i)}}{\varepsilon} \right)  
\end{align*}
Let
$$I_N:= \frac{1}{N} \sum_{\substack{m \neq l,\\ \norm{z_m - z_l}_{\infty} \leq \frac{s}{N^{1/d}} + \varepsilon}} \prod_{i=1}^d \left(1 - \frac{z_m^{(i)}-z_l^{(i)}}{\varepsilon} \right)$$
By the Koksma-Hlawka inequality, see e.g. \cite{Nie92}, Theorem~2.9, we have
$$\left| I_N - \int_{\norm{z_m-z_l}_{\infty} \leq \varepsilon + s/N^{1/d}} \prod_{i=1}^d \left(1 - \frac{x^{(i)}-z_l^{(i)}}{\varepsilon} \right) \mathrm{d} x\right| \leq C \cdot D_N(z_n)$$
with $C$ being a constant independent of $N$. Again by exploiting the fact that we have $\norm{\cdot}_{\infty}$ as norm the integral can be calculated component per component as
$$\left| I_N - \varepsilon^{d} \left( 1- \frac{s^2}{\varepsilon^2 N^2}\right)^d \right| \leq C \cdot D_N(z_n),$$
which is (very importantly) independent of $l$. Since $(\vec{z}_n)_{n \in \mathbb{N}}$ is a low-discrepancy sequence we can hence in summary deduce
\begin{align*}
\mathbb{E}(F_{N,d}(s)) & = \frac{1}{N} \cdot N \cdot \frac{(2s)^d}{\varepsilon^d} \cdot \varepsilon^d \cdot \left( \left(1 - \frac{s^2}{\varepsilon^2 N^2}\right)^d + \mathcal{O}\left( \frac{\log(N)}{N} \right)\right),
\end{align*}
which has limit $(2s)^d$ for $N \to \infty$, as claimed.
\end{proof}
All the preparatory results have now paved the way to the proof of Theorem~\ref{thm:main}.
\begin{proof}[Proof of Theorem~\ref{thm:main}]
As we want to apply Chebychev's inequality, we need to calculate the variance of $F_N(s)$. We closely follow here the approach in \cite{Mar07} although the technical details of our arguments are a lot more involved. Using the Poisson summation formula we write
\begin{align*}
F_{N.d}(s) & = \frac{1}{N} \sum_{1 \leq k \neq l \leq N} \sum_{\vec{q} \in \mathbb{Z}^{d}} I\left( \frac{(\varepsilon(\vec{X}_k-\vec{X}_l) + \left\{(k-l)\vec{\alpha}\right\}+\vec{q})N^{1/d}}{2s}\right)\\
& = \frac{(2s)^{d}}{N^2} \sum_{1 \leq k \neq l \leq N} \sum_{\vec{r} \in \mathbb{Z}^{d}}
\mathcal{F}I \left( \frac{2s\varepsilon\vec{r}}{N^{1/d}} \right) e(\langle{\vec{r}},(\varepsilon(\vec{X}_k-\vec{X}_l)+\left\{(k-l)\vec{\alpha}\right\})\rangle),
\end{align*}
where $\mathcal{F}I(\xi) = \prod_{i=1}^d \frac{sin(\pi \xi_i)}{\pi \xi_i}$ for $\xi_i \neq 0$ and $\mathcal{F}I(\vec{0})=1$ else is the Fourier transform of the indicator function $I(\cdot)$ of the interval $[-1/2,1/2]^{d}$. Thus, it follows analogously to \cite{HKL19} that
\begin{align*}
    & \mathbb{E}\left[ \left(F_{N.d}(s) - \mathbb{E}(F_{N.d}(s))  \right)^2 \right]\\
    & = \frac{(2s)^{2d}}{N^4} \sum_{\vec{r},\vec{r'} \in \mathbb{Z}^{d} \setminus \{\vec{0}\}} \sum_{\substack{1\leq k,l,m,n\leq N,\\ k \neq l, m \neq n}}  \mathcal{F}I \left( \frac{2s\varepsilon\vec{r}}{N^{1/d}} \right) \mathcal{F}I \left( \frac{2s\varepsilon\vec{r'}}{N^{1/d}} \right)\\
    & \qquad \times \mathbb{E}[e(\langle \vec{r}, (\varepsilon(\vec{X}_k-\vec{X}_l) + \{k\vec{\alpha}\} - \{l\vec{\alpha}\})\rangle +  \langle \vec{r'}, (\varepsilon(\vec{X}_m-\vec{X}_n) + \{m\vec{\alpha}\} - \{n\vec{\alpha}\})\rangle)]\\
     & = \frac{(2s)^{2d}}{N^4} \sum_{\vec{r},\vec{r'} \in \mathbb{Z}^{d} \setminus \{\vec{0}}\} \sum_{\substack{1\leq k,l,m,n\leq N,\\ k \neq l, m \neq n}}  \mathcal{F}I \left( \frac{2s\varepsilon\vec{r}}{N^{1/d}} \right) \mathcal{F}I \left( \frac{2s\varepsilon\vec{r'}}{N^{1/d}} \right)\\
    & \qquad \times \mathbb{E}[e(\langle \vec{r}, \varepsilon(\vec{X}_k-\vec{X}_l)\rangle +  \langle \vec{r'}, \varepsilon(\vec{X}_m-\vec{X}_n)\rangle)]\\
    & \qquad \times e(\langle \vec{r},(\{k\vec{\alpha}\} - \{l\vec{\alpha}\})\rangle +  \langle\vec{r'}, (\{m\vec{\alpha}\} - \{n\vec{\alpha}\})\rangle).
\end{align*}
If $\varepsilon \in \mathbb{Z}$ then the proof is finished because $\{k \vec{\alpha} + \varepsilon \vec{X}_{k} \}_{k \in \mathbb{N}}$ would be uniformly distributed on $[0,1]^d$ and the proof in \cite{HKL19} applies. Therefore, we may without loss of generality assume that $\varepsilon \notin \mathbb{Z}$. We then split up the sum into several parts. According to Lemma~\ref{lem:integral}, there are two main cases to distinguish.\\[12pt]
\underline{1. Case $r_i \neq \pm r'_i$ for all $1 \leq i \leq d$:}\\[12pt]
For $r_i \neq \pm r'_i$ for all $1 \leq i \leq d$ we split up the sum into the cases which occur in Lemma~\ref{lem:integral}. At first, we consider the case $k=m, l=n$, i.e. the sum
\begin{align*}
\frac{(2s)^{2d}}{N^4} & \sum_{\substack{\vec{r},\vec{r'} \in \mathbb{Z}^{d} \setminus \{\vec{0}\}\\ r_i \neq r_i'}}  \sum_{\substack{1\leq k,l,m,n\leq N,\\ m = k \neq l = n}}  \mathcal{F}I \left( \frac{2s\varepsilon\vec{r}}{N^{1/d}} \right) \mathcal{F}I \left( \frac{2s\varepsilon\vec{r'}}{N^{1/d}} \right)\\
    & \qquad \times \mathbb{E}[e(\langle\vec{r}, \varepsilon(\vec{X}_k-\vec{X}_l) \rangle+  \langle \vec{r'}, \varepsilon(\vec{X}_m-\vec{X}_n)\rangle)]\\
    & \qquad \times e(\langle\vec{r},(\{k\vec{\alpha}\} - \{l\vec{\alpha}\})\rangle +  \langle \vec{r'}, (\{m\vec{\alpha}\} - \{n\vec{\alpha}\})\rangle).
\end{align*}
We now focus the inner sum over $k=m, l=n$. By Lemma~\ref{lem:integral}, we know that the expected value is independent of the explicit values of $k,l,m$ and $n$. Therefore, we can apply Lemma~\ref{lem:upperbound} and bound the part stemming from the two Kronecker sequences by $C N^{3/2} |r_1|^{1/4} |r_1'|^{1/4}$ because $\min_{i \leq 1 \leq d} r_i \leq r_1$. If we furthermore write out the functions $\mathcal{F}I(\cdot)$, which are both independent of $k,l,m$ and $n$, we obtain by Lemma~\ref{lem:integral} for the inner sum  
\begin{align*}
    & = \left|CN^{7/2} |r_1|^{1/4} |r_1'|^{1/4} \prod_{i=1}^d \frac{\sin\left( \frac{2\pi s r_i\varepsilon}{N} \right)\sin\left( \frac{2\pi sr_i'\varepsilon}{N} \right)}{2\pi^2 sr_i\varepsilon 2\pi^2 sr'_i\varepsilon} \frac{\sin\left((r_i+r'_i)\pi \varepsilon)^2 \right)}{\pi^2 \varepsilon^2 (r_i+r'_i)^2}\right|\\
    & \leq C N^{7/2} \frac{1}{|r_1'|^{3/4}} \frac{1}{|r_1|^{3/4}}  \frac{1}{|r_1+r_1'|^2} \prod_{i=2}^d\frac{1}{|r'_i|} \frac{1}{|r_i|}  \frac{1}{|r_i+r_i'|^2}
\end{align*}
We can now apply at first Lemma~\ref{lem:shortsums} to bound the sum over $r_1$ by $\frac{C}{|r_1'|^{3/4}}$ and then we can take the sum over $r'_1$. Afterwards we sum over the remaining $r_j$ and $r'_j$. Including the factor $\frac{(2s)^{2d}}{N^4}$ we achieve a bound of $C N^{-1/2}$ for the complete sum under consideration.\\[12pt]
Next we turn to the case $k \neq m,n$ and $l \neq m,n$. Again by Lemma~\ref{lem:integral} the expected value is independent of $k,l,m,n$ and the sum reduces to
\begin{align*}
\frac{(2s)^{2d}}{N^4} & \sum_{\substack{\vec{r},\vec{r}' \in \mathbb{Z}^{d} \setminus \{\vec{0}\}\\ r_i \neq \pm r'_i}} \sum_{\substack{1\leq k,l,m,n\leq N,\\ \#\{k,l,m,n\} =  4}}  \prod_{i=1}^d \mathcal{F}I \left( \frac{2sr_i\varepsilon}{N^{1/d}} \right) \mathcal{F}I \left( \frac{2sr'_i\varepsilon}{N^{1/d}} \right) \frac{\sin(r_i\pi \varepsilon)^2}{r_i^2 \pi^2 \varepsilon^2}  \frac{\sin(r_i'\pi \varepsilon)^2}{r_i^{'2} \pi^2 \varepsilon^2}  \\
    & \times e(\langle \vec{r},(\{k\vec{\alpha}\} - \{l\vec{\alpha}\})\rangle +\langle\vec{r'}, (\{m\vec{\alpha}\} - \{n\vec{\alpha}\})\rangle).
\end{align*}
In the inner sum, only the parts stemming from the Kronecker sequence depend on $k,l,m$ and $n$ and by Lemma~\ref{lem:upperbound} the entire product can be bounded by $C \cdot N^3 \cdot |r_1|^{1/2}|r_1'|^{1/2}$ because $\# \{k,l,m,n\} = 4$. Therefore, by Lemma~\ref{lem:shortsums} we obtain
\begin{align*}
\frac{(2s)^{2d}}{N^4} &  \sum_{\substack{\vec{r},\vec{r}' \in \mathbb{Z}^{d} \setminus \{\vec{0}\}\\ r_i \neq \pm r'_i}} \sum_{\substack{1\leq k,l,m,n\leq N,\\ \#\{k,l,m,n\} =  4}}  \left| \prod_{i=1}^d \mathcal{F}I \left( \frac{2sr_i\varepsilon}{N^{1/d}} \right) \mathcal{F}I \left( \frac{2sr'_i\varepsilon}{N^{1/d}} \right) \frac{\sin(r_i\pi \varepsilon)^2}{r_i^2 \pi^2 \varepsilon^2}  \frac{\sin(r_i'\pi \varepsilon)^2}{r_i^{'2} \pi^2 \varepsilon^2}  \right.\\
    & \qquad \times e(\langle \vec{r},(\{k\vec{\alpha}\} - \{l\vec{\alpha}\})\rangle +  \langle\vec{r'}, (\{m\vec{\alpha}\} - \{n\vec{\alpha}\})\rangle) \bigg\rvert \\
%\leq & \tilde{C} \cdot \frac{1}{N} \sum_{\substack{r,r' \in \mathbb{Z} \setminus \{0\}\\ r \neq \pm r'}}  \frac{1}{|r|^{3/2}} \frac{1}{|r|'^{3/2}}.\\
\leq & \tilde{C} \zeta\left( \frac{3}{2} \right)^2  \frac{1}{N}
\end{align*}
All the remaining cases in Lemma~\ref{lem:integral} can be treated by combining the arguments of the two cases which we discussed here in detail. Summing up, it follows that
\begin{align*}
\frac{(2s)^{2d}}{N^4} & \sum_{\substack{\vec{r},\vec{r'} \in \mathbb{Z}^{d} \setminus \{\vec{0}\}\\ r_i \neq r_i'}}  \sum_{\substack{1\leq k,l,m,n\leq N,\\ k \neq l, m \neq n}}  \mathcal{F}I \left( \frac{2s\varepsilon\vec{r}}{N^{1/d}} \right) \mathcal{F}I \left( \frac{2s\varepsilon\vec{r'}}{N^{1/d}} \right)\\
    & \qquad \times \mathbb{E}[e(\langle \vec{r}, \varepsilon(\vec{X}_k-\vec{X}_l)\rangle +  \langle \vec{r'}, \varepsilon(\vec{X}_m-\vec{X}_n)\rangle)]\\
    & \qquad \times e(\langle \vec{r},(\{k\vec{\alpha}\} - \{l\vec{\alpha}\})\rangle + \langle \vec{r'}, (\{m\vec{\alpha}\} - \{n\vec{\alpha}\})\rangle)\\
    & = \mathcal{O}\left( \frac{1}{N^{1/2}} \right).
\end{align*}\\[12pt]
\underline{2. Case $r_i = \pm r'_i$ for some indices $1 \leq i \leq d$}:\\[12pt]
In this case we split up the indices into two subsets, namely $I_1:= \{i : r_i \neq \pm r_i'\}$ and $I_2 := \{1,\ldots,d\} \setminus I_1$. Let us rewrite the entire sum as
\begin{align*}
     & \frac{(2s)^{2d}}{N^4} \sum_{(I_1,I_2)} \sum_{\substack{\vec{r},\vec{r'} \in \mathbb{Z}^d \setminus \{\vec{0}\}\\ r_i \neq \pm r_i, i \in I_1\\r_i = \pm r_i, i \in I_2 }} \sum_{\substack{1\leq k,l,m,n\leq N,\\ k \neq l, m \neq n}}  \mathcal{F}I \left( \frac{2s\varepsilon\vec{r}}{N^{1/d}} \right) \mathcal{F}I \left( \frac{2s\varepsilon\vec{r}'}{N^{1/d}} \right)\\
    & \qquad \times \mathbb{E}[e(\langle\vec{r}, \varepsilon(\vec{X}_k-\vec{X}_l)\rangle +  \langle\vec{r'}, \varepsilon(\vec{X}_m-\vec{X}_n)\rangle)]\\
    & \qquad \times e(\langle\vec{r},(\{k\vec{\alpha}\} - \{l\vec{\alpha}\})\rangle +  \langle\vec{r'}, (\{m\vec{\alpha}\} - \{n\vec{\alpha}\})\rangle).
\end{align*}
Note that the number of possible combinations $(I_1,I_2)$ is dependent on $d$ but independent of $N$. This implies that we only have to make sure that the sum under consideration converges for every possible combination $(I_1,I_2)$ but there is no need to count the number of instances such a combination occurs (the number of combinations is part of the constant $C$). In the following, we will show that the order of convergence does not deteriorate from splitting up the indices.\\[12pt]
Now we consider the indices in $I_2$ and split up the inner sum (over $k,l,m,n$). If $r_i=r_i'$ and $k=m, l=n$ or if $r_i=-r_i'$ and $k=n, l=m$ respectively, then the sum over $r_i$ reduces in the case $r_i=r_i'$ to
\begin{align*}
& \sum_{r_i = r_i' \in \mathbb{Z} \setminus \{0\}} \sum_{\substack{1\leq k,l,m,n\leq N,\\ m = k \neq l = n}} \mathcal{F}I \left( \frac{2sr_i\varepsilon}{N^{1/d}} \right) \mathcal{F}I \left( \frac{2sr'_i\varepsilon}{N^{1/d}} \right)\\
    & \qquad \times \mathbb{E}[e(r_i \varepsilon(X_k^{(i)}-X_l^{(i)}) +  r'_i \varepsilon(X_m^{(i)}-X_n^{(i)}))]
\end{align*}
where $X_{\boldsymbol{\cdot}}^{(i)}$ denotes the $i$-th component of the respective random vector (and similarly in the case $r_i = -r_i'$). As the expected value in the expression is equal to $1$ in this case, we see that the sum is of order $\mathcal{O}(N^{2/d})$, i.e. an index in $I_2$ (corresponding to a pair $(r_i,r_i')$) adds an multiplicative factor $N^{2/d}$.\\[12pt]
Moreover, we have seen in the first case that each pair of indices in $I_1$ individually adds a multiplicative factor $N^{2/d}$ as well and that there is a collective factor $N^{3/2}$ for all indices in $I_1$. Summing up, this subsum over $k=m$ and $l=n$ converges like $\mathcal{O}(N^{1/2})$ and so does the subsum over $k=n,l=m$.\\[12pt] 
If $\# \{k,l,m,n\} = 4,$ then the proof from the case $r_i \neq \pm r_i'$ can be applied almost verbatim with the only difference that the double sum over $r_i,r_i'$ reduces to a single sum for indices in $I_1$. By the convergence of $\sum_{r_i \neq 0} \frac{1}{|r_i|^3}$, also here an order of convergence $\frac{1}{N^{2/d}}$ can be achieved for each index in $I_1$. Noticing that also for each index in $I_2$ there is a contribution of $\frac{1}{N^{2/d}}$, the total order of convergence $\mathcal{O}(1/N)$ can be achieved for this subsum.\\[12pt]
If we consider the case $k=m, l \neq n$, then the expected value from Lemma~\ref{lem:integral} for $r_i=-r'_i$ takes the form $\frac{1}{r_i^2\varepsilon^2} \sin(r_i \pi \varepsilon)^2$. By Lemma~\ref{lem:upperbound} and Lemma~\ref{lem:doublesum} (with $\delta = 1/d$ if $d \geq 2$) we see that the indices in $I_1$ yield a collective multiplicative factor of order $N^{2/d}$ and a collective factor $N^{3/2}$. Since the additional $r_i$ from Lemma~\ref{lem:upperbound} has already been taken into account, the indices in $I_1$ (respectively the pairs $(r_i,r_i')$) do not yield a collective factor but only an individual one of order $N^{2/d}$. In total, we end up with a order of convergence $\mathcal{O}\left(\frac{1}{N^{1/2}}\right)$ for this subsum.\\[12pt] 
For all other cases considered in Lemma~\ref{lem:integral}, it follows by a similar argument that the sum has order of convergence $\mathcal{O}\left(\frac{1}{N^{1/2}}\right)$. Therefore, the entire sum is of order $N^{1/2}$.\\[12pt]
This puts us into the position to finally apply Chebyshev's inequality and get for arbitrary $\delta > 0$ that
$$\mathbb{P} \left( \left| F_{N,d}(s) - (2s)^d \right| \geq \delta \right) \leq \frac{C}{\delta^2 N^{\tfrac{1}{2}}}$$
with $C$ independent of $N$, i.e. convergence in probability of $F_{N,d}(s)$ to $(2s)^d$. In order to get almost sure convergence, the remainder of the proof is analogously to the one of Theorem~1 in \cite{HKL19} and the presentation of the arguments is therefore omitted here. %For the remainder of the proof, we now use another idea from \cite{HKL19}. We fix a $\gamma \in 2\mathbb{N}$ and define a sequence $N_M := M^{2+\gamma}$ for $M \in \mathbb{N}$. Then the first Borelli-Cantelli lemma yields
%$$\lim_{M \to \infty} F_{N_M}(s) = 2s, \qquad \textrm{almost surely}.$$
%The claim of the theorem now follows for $N$ with $N_M \leq N \leq N_{M+1}$ by the trivial inequalities
%$$N_M F_{N_M}\left(\frac{N_M}{N_{M+1}} s \right) \leq N F_N(s) \leq N_{M+1} F_{N_{M+1}}\left(\frac{N_{M+1}}{N_{M}} s \right)$$
%and $N_{M+1}/N_M \to 1$ as $M \to \infty$.
\end{proof}

\bibliographystyle{alpha}
\bibdata{literatur}
\bibliography{literatur}

\textsc{Rosenheim Technical University of Applied Sciences, Hochschulstra\ss{}e 1, D-83024 Rosenheim,} \texttt{anja.schmiedt@th-rosenheim.de}

\textsc{Ruhr West University of Applied Sciences, Duisburger Str. 100, D-45479 M\"ulheim an der Ruhr,} \texttt{christian.weiss@hs-ruhrwest.de}

\end{document}